\documentclass[12pt]{article}
\usepackage[T1]{fontenc}
\usepackage[utf8]{inputenc}
\usepackage{lmodern}
\usepackage[a4paper,margin=1in]{geometry}
\usepackage{amsmath,amssymb,amsthm,mathtools}
\usepackage{authblk}
\usepackage{microtype}
\usepackage[hidelinks]{hyperref}
\newtheorem{theorem}{Theorem}[section]
\newtheorem{lemma}[theorem]{Lemma}
\newtheorem{claim}[theorem]{Claim}
\DeclareMathOperator{\AR}{AR}
\DeclareMathOperator{\exdir}{ex_{dir}}
\newcommand{\dP}{\overrightarrow{P}}
\newcommand{\sK}{\overleftrightarrow{K}}
\title{Spanning Even Linear Forests with Isolated Edges\\
	Share the Anti-Ramsey Numbers of Perfect Matchings}
\author[,a]{Ali Ghalavand\thanks{Corresponding author.
		E-mail: \texttt{alighalavand@himis-sz.cn}}}

\affil[a]{Hetao Institute of Mathematics and Interdisciplinary
	Sciences (HIMIS), Shenzhen 518000, Guangdong, P. R. China}

\date{}
\begin{document}
\maketitle

\begin{abstract}
The anti-Ramsey number $\AR(n,F)$ is the maximum number of colors in an
edge-coloring of the complete graph $K_n$ containing no rainbow copy of $F$.
I prove that every spanning linear forest on an even number $n\ge6$ of
vertices whose components have even orders and which has an isolated edge
has the same anti-Ramsey number as a perfect matching on $n$ vertices.
The orders of the longer path components may be different.
\end{abstract}



\section{Introduction}

All undirected graphs considered in this paper are finite and simple.
For a graph $G$, its vertex and edge sets are denoted by $V(G)$ and $E(G)$,
and its order and size are $|V(G)|$ and $e(G)=|E(G)|$, respectively.
The complete graph on $n$ vertices is denoted by $K_n$, and $P_n$ denotes
the path on $n$ vertices, so that $P_n$ has $n-1$ edges.
The symbol $G\cup H$ denotes the vertex-disjoint union of two graphs,
and $kG$ denotes the disjoint union of $k$ copies of $G$.
We write $G\cong H$ for graph isomorphism and $G\subseteq H$ when $H$
contains a subgraph isomorphic to $G$; this subgraph need not be induced.
A linear forest is a graph whose components are paths. An isolated edge
means a component isomorphic to $P_2$. We call a linear forest \emph{even} if every
component has an even number of vertices. A perfect matching is a set of
pairwise vertex-disjoint edges covering all vertices. When a matching $M$ is contained in
a graph $G$, we also use $M$ for its edge set, and $G-M$ denotes deletion
of those edges while retaining all vertices.

An edge-coloring $c$ assigns a color to each edge of a graph; it is not
required to be proper. A subgraph is \emph{rainbow} if all its edges have
distinct colors. For a graph $F$ of order at most $n$, the anti-Ramsey number
$\AR(n,F)$ is the maximum number of colors in an edge-coloring of $K_n$
with no rainbow copy of $F$. A \emph{representative graph} of a coloring is a spanning subgraph obtained
by choosing exactly one edge of each color. It is therefore rainbow and
has as many edges as the coloring has colors.

Anti-Ramsey theory was introduced by Erd\H{o}s, Simonovits, and S\'os
\cite{ESS}. Classical results on paths and cycles were obtained by
Simonovits and S\'os \cite{SS} and Montellano-Ballesteros and
Neumann-Lara \cite{MBNL}, respectively. For matchings, Schiermeyer
\cite{Sch} determined the rainbow threshold for sufficiently large
host orders. Chen, Li, and Tu \cite{CLT} and Fujita, Kaneko,
Schiermeyer, and Suzuki \cite{FKSS} determined the anti-Ramsey numbers
of $tP_2$ for $n\ge2t+1$. Perfect matchings form the remaining
spanning case, which was settled by Haas and Young \cite{HY}.
Jahanbekam and West \cite{JW} studied the related packing problem of
forcing edge-disjoint rainbow spanning subgraphs, including perfect
matchings, cycles, and trees. These results provide a broader setting
for the study of rainbow structures that cover the entire host graph.

Linear forests have also received substantial attention in anti-Ramsey
theory. Bialostocki, Gilboa, and Roditty \cite{BGR} determined the
anti-Ramsey numbers of graphs with at most four edges, including
$P_3\cup P_2$, $P_3\cup2P_2$, and $P_4\cup P_2$.
Gilboa and Roditty \cite{GR} developed methods for studying the
disjoint union of a fixed graph with copies of $P_2$ or $P_3$.
For sufficiently large host orders, they obtained exact values for
several short-path forests and bounds for a path together with a
matching. Fang, Gy\H{o}ri, Lu, and Xiao \cite{FGLX} obtained bounds for
linear forests and determined the exact anti-Ramsey number of $2P_4$.
Xie and Yuan \cite{XY} subsequently determined the exact anti-Ramsey
numbers of fixed linear forests containing an even-order path for
sufficiently large complete host graphs, and characterized the extremal
colorings. The spanning case, in which the forest uses every vertex of
the host graph, requires an exact treatment at the smallest possible
host order.

Recent work has focused on forests formed by short paths and isolated
edges. He and Jin \cite{HJ} determined anti-Ramsey numbers for
$P_3\cup tP_2$ and $2P_3\cup tP_2$ in specified ranges of the host order.
Jie, He, and Jin \cite{JHJ} determined anti-Ramsey numbers of
$kP_3\cup tP_2$ for $k\ge2$, $t\ge(k^2-k+4)/2$, and
$n\ge3k+2t+1$.
Jie and Jin \cite{JJ} investigated spanning forests of the form
$kP_3\cup tP_2$ under a condition relating the numbers of the two types
of components. Ghalavand and Li \cite{GL} determined
$\AR(3k+2t,kP_3\cup tP_2)$ for all $k\ge1$ and $t\ge2$, without
an additional relation between $k$ and $t$.
These spanning results concern forests with three-vertex path components.
For even components, Ghalavand, Jie, Jin, Li, and Pan
\cite{GJJLP} established the following result for spanning forests
consisting of paths on four vertices and isolated edges.

\begin{theorem}{\rm\cite{GJJLP}}\label{lem:published-p4}
If $k,t\ge1$, then
\[
 \AR(4k+2t,kP_4\cup tP_2)=\AR(4k+2t,(2k+t)P_2).
\]
\end{theorem}

The present paper extends Theorem~\ref{lem:published-p4} to arbitrary even
component orders. More precisely, we prove that if
\[
 F\cong P_{2a_1}\cup\cdots\cup P_{2a_q}\cup tP_2,
 \qquad a_i\ge2,\quad t\ge1,
\]
where $q\ge0$ and $n=2\sum_{i=1}^{q}a_i+2t\ge6$, then
$\AR(n,F)=\AR(n,\frac{n}{2}P_2)$. Here the union and sum over the longer
paths are empty when $q=0$, so this statement also includes perfect
matchings. No equality among the integers $a_1,\ldots,a_q$ is assumed.

Every even linear forest $F$ on $2s$ vertices with an isolated edge satisfies
$sP_2\subseteq F\subseteq P_{2s-2}\cup P_2$. Thus it is enough to determine
$\AR(2s,P_{2s-2}\cup P_2)$. We use a representative graph containing a
rainbow perfect matching and select endpoint labels that maximize the
number of arcs in an associated digraph. Comparing this labeling with
the result of swapping the labels on a single matching edge shows that
at least half of the nonmatching edges contribute arcs. The directed
extremal results of Zhou and Li \cite{ZL}, together with the auxiliary
bounds proved here, then yield a long directed path. Expanding its
vertices back into matching edges produces a rainbow path disjoint
from one unused matching edge.

We specify the directed notation used in that reduction. A digraph $D$
has a vertex set $V(D)$ and an arc set $A(D)$, with no loops and at most
one arc in each direction between two distinct vertices. Opposite arcs
are allowed. Its number of arcs is $a(D)=|A(D)|$.
The notation $u\to v$ means that $(u,v)\in A(D)$.
For a vertex $v$, its outdegree and indegree are $d_D^+(v)$ and $d_D^-(v)$,
and its total degree is $d_D(v)=d_D^+(v)+d_D^-(v)$.
For $U\subseteq V(D)$, the induced subdigraph on $U$ is denoted by $D[U]$;
$D-v$ means $D[V(D)\setminus\{v\}]$.
For disjoint sets $U,W$, let $a_D(U,W)$ denote the total number of arcs
with one endpoint in $U$ and the other in $W$, counting both directions.
When $P$ is a path, $V(P)$ is its vertex set.
The complete symmetric digraph on $r$ vertices, denoted by $\sK_r$,
contains both arcs between every pair of distinct vertices.
An isolated directed two-cycle is a component on two vertices containing
both opposite arcs and having no arcs to or from other vertices.
A directed Hamilton cycle visits every vertex
exactly once before returning to its starting vertex.
The directed path on $n$ distinct vertices, all of whose arcs point along
the path, is denoted by $\dP_n$.
Finally, $\exdir(s,\dP_n)$ is the maximum number of arcs in a digraph on
$s$ vertices containing no copy of $\dP_n$.
The notation $\lfloor x\rfloor$ denotes the integer part of $x$.

Section~2 contains the published results used in the proof and the
auxiliary directed-path bounds. Section~3 determines $\AR(2s,P_{2s-2}\cup P_2)$
and deduces the result for spanning even linear forests. Section~4
concludes the paper.

\section{Auxiliary lemmas}

In this section, we collect the auxiliary results needed to prove our main theorems. We recall relevant results from the literature and establish the additional lemmas used in our proofs.

\begin{lemma}[Haas and Young~\cite{HY}]
\label{lem:HY} If $s\geq2$ is an integer, then 
\[
\AR(2s,sP_2)=
 \begin{cases}
  3,&s=2,\\[1mm]
  \dfrac{(s-2)(3s+1)}{2}+1,&3\le s\le6,\\[2mm]
  (s-2)(2s-3)+2,&s\ge7.
 \end{cases}
\]
\end{lemma}

The following Hamilton-cycle bound and its equality characterization are
special cases of the directed-cycle results of Zhou and Li \cite{ZL}.

\begin{lemma}[Zhou and Li~\cite{ZL}]\label{lem:cycle-density}
Let $H$ be a digraph on $r\ge3$ vertices without a directed Hamilton cycle.
Then $a(H)\le(r-1)^2$. If $r\ge4$ and $a(H)=(r-1)^2$, then $H$ is obtained
from $\sK_{r-1}$ by adding a vertex $z$ and either all arcs from $z$ to
the other vertices or all arcs from the other vertices to $z$.
\end{lemma}

\begin{lemma}\label{lem:insertion}
Let $P=v_1\to v_2\to\cdots\to v_r$ be a longest directed path in a
digraph $D$. For every $x\in V(D)\setminus V(P)$,
\[
 a_D(\{x\},V(P))\le r-1.
\]
\end{lemma}

\begin{proof}
	Since \(P\) is a longest directed path, neither \(x\to v_1\) nor
	\(v_r\to x\) can be present; otherwise \(P\) could be extended to a
	directed path on \(r+1\) vertices.
	It remains to consider the arcs
	\[
	v_i\to x \quad (1\le i\le r-1)
	\qquad\text{and}\qquad
	x\to v_{i+1} \quad (1\le i\le r-1).
	\]
	For each \(i=1,\ldots,r-1\), at most one of the two arcs
	\[
	v_i\to x,\qquad x\to v_{i+1}
	\]
	can be present. Indeed, if both were present, then
	\[
	v_1\to\cdots\to v_i\to x\to v_{i+1}\to\cdots\to v_r
	\]
	would be a directed path on \(r+1\) vertices, contradicting the maximality
	of \(P\).	
	These \(r-1\) pairs are pairwise disjoint, and every remaining possible
	arc between \(x\) and \(V(P)\) belongs to exactly one of them. Hence the
	number of arcs between \(x\) and \(V(P)\) is at most \(r-1\).
\end{proof}

\begin{lemma}[Zhou and Li~\cite{ZL}]\label{lem:short-path}
For every integer $s\ge3$,
\[
 \exdir(s,\dP_3)=\left\lfloor\frac{s^2}{4}\right\rfloor.
\]
\end{lemma}

\begin{lemma}\label{lem:two-outside}
	Let $D$ be a digraph on $r+2$ vertices, where $r\ge4$.
	If a longest directed path in $D$ has $r$ vertices, then
	$a(D)\le r^2$.
\end{lemma}

\begin{proof}
	Choose a longest directed path $P$ in $D$, and put
	$H=D[V(P)]$ and $W=V(D)\setminus V(P)=\{x,y\}$.
	By Lemma~\ref{lem:insertion},
	\begin{equation}\label{eq:cross}
		a_D(W,V(P))\le 2(r-1).
	\end{equation}
	Moreover,
	\[
	a(D)=a(H)+a_D(W,V(P))+a(D[W]),
	\qquad a(D[W])\le2.
	\]
	
	Suppose first that $H$ contains a directed Hamilton cycle.
	No arc can join $W$ to $V(H)$: an arc directed from an
	outside vertex to the cycle produces a directed path starting
	at that vertex and then visiting all $r$ cycle vertices;
	an arc in the opposite direction similarly produces a
	directed path ending at the outside vertex.
	Either construction gives a directed path on $r+1$ vertices,
	contrary to the choice of $P$. Thus
	\[
	a(D)\le r(r-1)+2\le r^2.
	\]
	
	We may therefore assume that $H$ contains no directed
	Hamilton cycle. Lemma~\ref{lem:cycle-density} gives
	$a(H)\le(r-1)^2$.
	If $a(H)\le(r-1)^2-1$, or if
	$a(H)=(r-1)^2$ and $a(D[W])\le1$, then
	\[
	a(H)+a(D[W])\le(r-1)^2+1.
	\]
	Together with \eqref{eq:cross}, this yields
	\[
	a(D)\le(r-1)^2+1+2(r-1)=r^2.
	\]
	
	It remains to consider the case
	$a(H)=(r-1)^2$ and $a(D[W])=2$.
	In particular, both $x\to y$ and $y\to x$ are present.
	By the equality statement in Lemma~\ref{lem:cycle-density},
	there is a vertex $z\in V(H)$ such that, writing
	$C=V(H)\setminus\{z\}$, we have
	$D[C]\cong\sK_{r-1}$, and either all arcs from $z$ to $C$
	or all arcs from $C$ to $z$ are present, with no arcs
	in the opposite direction.
	
	We first show that $a_D(W,C)=0$.
	Take $u\in W$, and let $w$ be the other vertex of $W$.
	If $u\to c$ is present for some $c\in C$, then
	$w\to u\to c$, followed by a directed Hamilton path
	of $D[C]$ starting at $c$, gives a directed path on
	$r+1$ vertices.
	If $c\to u$ is present, a directed Hamilton path of
	$D[C]$ ending at $c$, followed by $c\to u\to w$,
	gives the same contradiction.
	Hence no arc joins $W$ to $C$.
	
	Finally, suppose that all arcs between $z$ and $C$
	are directed from $z$ to $C$.
	For any $u\in W$, an arc $u\to z$ could be followed
	by a directed Hamilton path of $H$ starting at $z$,
	giving a directed path on $r+1$ vertices.
	Thus only $z\to x$ and $z\to y$ are possible.
	If instead all arcs are directed from $C$ to $z$,
	the analogous argument using a directed Hamilton path
	of $H$ ending at $z$ rules out $z\to x$ and $z\to y$.
	Consequently, in either case,
	\[
	a_D(W,\{z\})\le2.
	\]
	Therefore,
	\[
	a(D)
	=a(H)+a_D(W,C)+a_D(W,\{z\})+a(D[W])
	\le(r-1)^2+4<r^2,
	\]
	where the final inequality follows from $r\ge4$.
	This completes the proof.
\end{proof}

\begin{lemma}\label{lem:directed-bound}
For every integer $s\ge3$,
\[
\exdir(s,\dP_{s-1})\le
 \begin{cases}
  0,&s=3,\\
  4,&s=4,\\
  10,&s=5,\\
  (s-2)^2,&s\ge6.
 \end{cases}
\]
\end{lemma}
\begin{proof}
For $s=3$, avoiding $\dP_2$ means having no arcs. For $s=4$, the required
bound is Lemma~\ref{lem:short-path}.

Consider $s=5$, and let $D$ contain no $\dP_4$. A longest path has at most
three vertices. If it has at most two, Lemma~\ref{lem:short-path} gives
$a(D)\le6$. Otherwise, let $H$ be induced by its three vertices. If $H$
has a Hamilton cycle, the argument in Lemma~\ref{lem:two-outside} shows
that no arc joins it to the other two vertices, and $a(D)\le6+2=8$.
If it has no Hamilton cycle, Lemmas~\ref{lem:cycle-density} and
\ref{lem:insertion} give
\[
 a(D)\le(3-1)^2+2(3-1)+2=10.
\]

We prove the bound for $s\ge6$ by induction on $s$.
For the base case $s=6$, let $D$ contain no $\dP_5$, and let $r\le4$ be
the order of a longest path $P$. If $r=4$, Lemma~\ref{lem:two-outside}
gives $a(D)\le16$. If $r\le2$, Lemma~\ref{lem:short-path} gives
$a(D)\le9$. If $r=3$, put $H=D[V(P)]$. When $H$ has a Hamilton cycle,
there are no arcs between $H$ and its complement, so $a(D)\le6+6=12$.
Otherwise, Lemmas~\ref{lem:cycle-density} and \ref{lem:insertion} give
\[
 a(D)\le(3-1)^2+3(3-1)+3(3-1)=16=(6-2)^2.
\]
The last term bounds all arcs among the three vertices outside $P$.

Now let $s\ge7$ and assume the assertion for $s-1$. Let $D$ contain no
$\dP_{s-1}$, and choose a longest path $P=v_1\to\cdots\to v_r$.
Then $r\le s-2$. If $r=s-2$, Lemma~\ref{lem:two-outside} gives the
required bound directly. Suppose $r\le s-3$. The last vertex $v_r$ has
no outgoing arc to a vertex outside $P$. At most $2(r-1)$ arcs incident
with $v_r$ have their other endpoint on $P$, and at most $s-r$ have their
other endpoint outside $P$. Thus
\[
 d_D(v_r)\le2(r-1)+(s-r)=s+r-2\le2s-5.
\]
Since every path in $D$ has at most $s-3$ vertices, the digraph $D-v_r$
contains no $\dP_{s-2}$. It has $s-1$ vertices, so the induction hypothesis
gives $a(D-v_r)\le(s-3)^2$. Deleting a vertex removes exactly its total
degree many arcs. Consequently,
\[
 a(D)=a(D-v_r)+d_D(v_r)
 \le(s-3)^2+2s-5=(s-2)^2.
\]
This completes the induction.
\end{proof}
\begin{lemma}\label{lem:comparison}
For every integer $s\ge3$, $2\exdir(s,\dP_{s-1})+s\le \AR(2s,sP_2)$.
\end{lemma}
\begin{proof}
By Lemma~\ref{lem:HY}, the values of $\AR(2s,sP_2)$ for $s=3,4,5,6$ are
$6,14,25,39$, respectively. Lemma~\ref{lem:directed-bound} gives
\begin{align*}
& 2\exdir(3,\dP_{2})+3\le3\le6,\quad
2\exdir(4,\dP_{3})+4\le12\le14,\quad\\
&2\exdir(5,\dP_{4})+5\le25,\quad\quad\,\,
2\exdir(6,\dP_{5})+6\le38\le39.	
\end{align*}

For $s\ge7$, the two lemmas give
\[
 2\exdir(s,\dP_{s-1})+s\le2(s-2)^2+s
 =(s-2)(2s-3)+2=\AR(2s,sP_2).
\]
\end{proof}

\section{Main results}

We first consider $P_{2s-2}\cup P_2$, which contains every even linear
forest on $2s$ vertices having an isolated edge as a spanning subgraph.
The proof combines the directed-path bounds from Section~2 with an
endpoint labeling of a rainbow perfect matching that maximizes the arc
count of an auxiliary digraph. The result for all such forests then
follows by subgraph monotonicity and the perfect-matching formula.

\begin{theorem}\label{thm:central}
	For every integer $s\ge3$,
	\[
	\AR(2s,P_{2s-2}\cup P_2)=\AR(2s,sP_2).
	\]
\end{theorem}

\begin{proof}
	The forest $P_{2s-2}\cup P_2$ contains a perfect matching.
	Consequently, every edge-coloring of $K_{2s}$ without
	a rainbow $sP_2$ also avoids a rainbow $P_{2s-2}\cup P_2$.
	Thus
	\[
	\AR(2s,P_{2s-2}\cup P_2)\ge\AR(2s,sP_2).
	\]
	
	For the reverse inequality, consider an edge-coloring
	of $K_{2s}$ using $q>\AR(2s,sP_2)$ colors.
	By the definition of the anti-Ramsey number, this coloring
	contains a rainbow perfect matching $M$.
	
	Choose one representative edge of each color, choosing
	the corresponding edge of $M$ whenever that color appears
	on $M$. These choices are compatible because the edges
	of $M$ have distinct colors. The resulting spanning
	graph $G$ contains $M$, satisfies $e(G)=q$, and has
	pairwise distinctly colored edges.	
	Write
	\[
	M=\{u_i v_i:1\le i\le s\}.
	\]
	For each $i$, there are two ways to assign the endpoint
	labels $a_i,b_i$:
	\[
	(a_i,b_i)=(u_i,v_i)
	\quad\text{or}\quad
	(a_i,b_i)=(v_i,u_i).
	\]
	Hence there are $2^s$ possible complete labelings of $M$.
	For each labeling, define a digraph $D$ on $\{1,\ldots,s\}$ by
	\[
	i\to j
	\quad\Longleftrightarrow\quad
	b_i a_j\in E(G),
	\qquad i\ne j.
	\]
	The direction goes from the index of the $b$-endpoint
	to the index of the $a$-endpoint, regardless of the
	numerical order of the indices. More explicitly,
	for two distinct matching edges indexed by $i$ and $j$,
	\[
	\begin{array}{c|c}
		\text{Edge of }G & \text{Contribution to }D\\ \hline
		b_i a_j=a_j b_i & i\to j\\
		a_i b_j=b_j a_i & j\to i\\
		a_i a_j & \text{no arc}\\
		b_i b_j & \text{no arc}
	\end{array}
	\]
	Thus each nonmatching edge contributes exactly one arc
	when its endpoints have different label types, and
	no arc otherwise. Reversing the written order of its
	endpoints does not create another arc.
	
	The matching edges contribute no arcs because $i\ne j$
	is required. Moreover, for each ordered pair $(i,j)$,
	only the edge $b_i a_j$ can supply the arc $i\to j$.
	Therefore $D$ is a simple loopless digraph, and distinct
	contributing edges supply distinct arcs.
	Write $a(D)$ for its number of arcs.
	
	\begin{claim}\label{cl:max-labeling}
		There exists an endpoint labeling of $M$ for which
		\[
		a(D)\ge\frac{q-s}{2}
		>\frac{\AR(2s,sP_2)-s}{2}.
		\]
	\end{claim}
	
	\begin{proof}[Proof of the claim]
		Choose a labeling maximizing $a(D)$, which is possible
		because there are only $2^s$ labelings.
		
		For each $i$, let $t_i$ be the number of edges in
		$E(G)\setminus E(M)$ having one endpoint in
		$\{a_i,b_i\}$, and let $c_i$ count those edges that
		contribute arcs to $D$.
		Equivalently, $c_i$ is the number of arcs of $D$
		incident with $i$.
		
		Interchange the labels $a_i$ and $b_i$, keeping all
		other labels fixed, and let $D_i$ be the resulting
		digraph. Each nonmatching edge incident with
		$\{a_i,b_i\}$ has exactly one endpoint whose label
		type changes. After the interchange, each of the $c_i$ previously
		contributing edges has either two $a$-endpoints or two
		$b$-endpoints and therefore contributes no arc.
		Each of the remaining $t_i-c_i$ edges now has one
		$a$-endpoint and one $b$-endpoint and consequently
		contributes exactly one arc.
		All arcs with neither endpoint $i$ remain unchanged.
		Hence
		\[
		a(D_i)=a(D)-c_i+(t_i-c_i)
		=a(D)+t_i-2c_i.
		\]
		By the maximality of $a(D)$, we have $a(D_i)\le a(D)$.
		Therefore
		\[
		2c_i\ge t_i
		\qquad (1\le i\le s).
		\]
		
		Every arc of $D$ is counted at each of its two ends,
		so 
		\[\sum_{i=1}^s c_i=2a(D).\]
		Similarly, every nonmatching edge joins two distinct
		matching edges and is counted in the corresponding
		two quantities $t_i$. Thus
		\[
		\sum_{i=1}^s t_i
		=2|E(G)\setminus E(M)|
		=2(q-s).
		\]
		Summing $2c_i\ge t_i$ gives
		\[
		4a(D)
		=2\sum_{i=1}^s c_i
		\ge\sum_{i=1}^s t_i
		=2(q-s).
		\]
		Consequently,
		\[
		a(D)\ge\frac{q-s}{2}
		>\frac{\AR(2s,sP_2)-s}{2},
		\]
		where the strict inequality follows from
		$q>\AR(2s,sP_2)$.
	\end{proof}
	
	Fix the labeling supplied by Claim~\ref{cl:max-labeling}.
	By the directed-path arc bound and
	Lemma~\ref{lem:comparison}, every digraph on $s$ vertices
	containing no directed path on $s-1$ vertices has at most
	\[
	\frac{\AR(2s,sP_2)-s}{2}
	\]
	arcs. Our digraph exceeds this bound, so it contains
	a directed path
	\[
	i_1\to i_2\to\cdots\to i_{s-1},
	\]
	where the indices are distinct.
	
	For each $1\le h\le s-2$, the definition of $D$ gives
	\[
	b_{i_h}a_{i_{h+1}}\in E(G).
	\]
	Combining these connecting edges with the matching
	edges $a_{i_h}b_{i_h}$, for $1\le h\le s-1$, produces
	the path
	\[
	a_{i_1}b_{i_1}a_{i_2}b_{i_2}\cdots
	a_{i_{s-1}}b_{i_{s-1}}
	\]
	in $G$. Its $2s-2$ vertices are distinct because
	the indices are distinct and the matching edges
	are pairwise disjoint.
	
	Exactly one index $j$ is absent from
	$\{i_1,\ldots,i_{s-1}\}$. The matching edge $a_jb_j$
	is disjoint from this path. Together, the path and
	this unused matching edge form a copy of
	$P_{2s-2}\cup P_2$.
	
	All edges of this copy belong to $G$ and therefore
	have distinct colors. Thus the copy is rainbow.
	We have shown that every edge-coloring of $K_{2s}$
	using more than $\AR(2s,sP_2)$ colors contains a rainbow
	$P_{2s-2}\cup P_2$. Consequently,
	\[
	\AR(2s,P_{2s-2}\cup P_2)\le\AR(2s,sP_2).
	\]
	Together with the reverse inequality established above,
	this proves the theorem.
\end{proof}

\begin{theorem}\label{thm:even-forest}
Let $F$ be an even linear forest on $n=2s\ge6$ vertices, and suppose that
$F$ has at least one component isomorphic to $P_2$. Then
\[
 \AR(n,F)=\AR(n,sP_2)
 =\begin{cases}
 \dfrac{(s-2)(3s+1)}2+1,&3\le s\le6,\\[2mm]
 (s-2)(2s-3)+2,&s\ge7.
 \end{cases}
\]
\end{theorem}
\begin{proof}
Every even-order path has a perfect matching, obtained by pairing consecutive
vertices along the path. Taking these matchings over all components gives
$sP_2\subseteq F$.

Choose one $P_2$ component of $F$. Since the remaining
path components have total order $n-2$, partition the
vertices of $P_{n-2}$ into consecutive blocks whose orders
equal those of these components. Deleting the edges
between successive blocks leaves paths isomorphic to
the remaining components of $F$. Together with the
separate $P_2$, these paths form a spanning copy of $F$
in $P_{n-2}\cup P_2$.

If $H\subseteq H'$, every coloring avoiding a rainbow $H$ also avoids a
rainbow $H'$, so $\AR(n,H)\le\AR(n,H')$. Applying this observation to
the two containments and using Theorem~\ref{thm:central} gives
\[
 \AR(n,sP_2)\le\AR(n,F)
 \le\AR(n,P_{n-2}\cup P_2)=\AR(n,sP_2).
\]
The equality of the anti-Ramsey numbers follows, and their explicit values
are given by Lemma~\ref{lem:HY}.
\end{proof}

\section{Conclusion}

Every spanning even linear forest with an isolated edge has the same
anti-Ramsey number as a perfect matching on the same vertex set.
Thus, within this family, the value depends only on the total order,
regardless of the number and orders of the path components. The proof
reduces the problem to $P_{n-2}\cup P_2$ and combines a maximizing
endpoint labeling with directed-path bounds. This extends the known
result for $kP_4\cup tP_2$ to arbitrary even component orders.

\section*{Conflict of interest}
The authors declare that they have no conflict of interest.

\section*{Data availability} 

No data was used in this investigation.

\section*{Use of artificial intelligence}

The author used AI tools for algebraic checks, language editing,
locating relevant auxiliary results, and examining proof ideas.
The author reviewed and verified all mathematical statements,
calculations, proofs, and references and takes full responsibility
for the manuscript.

\end{document}